\documentclass{amsart}

\usepackage{amsmath}
\usepackage{amssymb}

\newcommand{\forces}{\Vdash}
\newcommand{\bV}{{\bf V}}
\newcommand{\lesdot}{\mathrel{\mathord{<}\!\!\raise
0.8 pt\hbox{$\scriptstyle\circ$}}}

\newcommand{\cf}{{\rm cf}\/}
\newcommand{\can}{{}^{\textstyle \omega}2}
\newcommand{\fs}{{}^{\textstyle\omega{>}}2}

\newcommand{\lh}{\ell g\/}
\newcommand{\rest}{{\restriction}}
\newcommand{\dom}{{\rm dom}}

\newcommand{\conc}{{}^\frown\!}
\newcommand{\vtl}{\vartriangleleft}
\newcommand{\vare}{\varepsilon}

\newcommand{\qfor}{{{\mathbb Q}^*_\mu(\bar{\varphi})}}
\newcommand{\qforname}{{{\mathbb Q}^*_\mu(\name{\bar{\psi}})}}
\newcommand{\cohenk}{{{\mathbb C}_\kappa}}

\newcommand{\add}{{\rm add}}
\newcommand{\cof}{{\rm cof}}
\newcommand{\cofin}{{\rm cofin}}

\newcommand{\cA}{{\mathcal A}}

\newcommand{\cB}{{\mathcal B}}
\newcommand{\bbB}{{\mathbb B}}
\newcommand{\bbC}{{\mathbb C}}

\newcommand{\cI}{{\mathcal I}}

\newcommand{\cH}{{\mathcal H}}

\newcommand{\cN}{{\mathcal N}}
\newcommand{\cM}{{\mathcal M}}

\newcommand{\cP}{{\mathcal P}}
\newcommand{\bbP}{{\mathbb P}}
\newcommand{\bbQ}{{\mathbb Q}}

\newcommand{\bP}{{\mathbb P}}

\newcommand{\cS}{{\mathcal S}}
\newcommand{\gt}{{\mathfrak t}}
\newcommand{\cX}{{\mathcal X}}

\newcommand{\bbY}{{\mathbb Y}}
\newcommand{\cZ}{{\mathcal Z}}

\newcount\skewfactor
\def\mathunderaccent#1#2 {\let\theaccent#1\skewfactor#2
\mathpalette\putaccentunder}
\def\putaccentunder#1#2{\oalign{$#1#2$\crcr\hidewidth
\vbox to.2ex{\hbox{$#1\skew\skewfactor\theaccent{}$}\vss}\hidewidth}}
\def\name{\mathunderaccent\tilde-3 }

\newtheorem{theorem}{Theorem}[section]
\newtheorem{claim}{Claim}[theorem]
\newtheorem{lemma}[theorem]{Lemma}
\newtheorem{proposition}[theorem]{Proposition}
\newtheorem{corollary}[theorem]{Corollary}
\newtheorem{observation}[theorem]{Observation}

\theoremstyle{definition}
\newtheorem{problem}[theorem]{Problem}

\newtheorem{definition}[theorem]{Definition}

\theoremstyle{remark}

\newtheorem{context}[theorem]{Context}

\begin{document}

\title{Around {\tt cofin}}

\author{Andrzej Ros{\l}anowski}
\address{Department of Mathematics\\
 University of Nebraska at Omaha\\
 Omaha, NE 68182-0243, USA}
\email{roslanow@member.ams.org}
\urladdr{http://www.unomaha.edu/logic}

\author{Saharon Shelah}
\address{Institute of Mathematics\\
 The Hebrew University of Jerusalem\\
 91904 Jerusalem, Israel\\
 and  Department of Mathematics\\
 Rutgers University\\
 New Brunswick, NJ 08854, USA}
\email{shelah@math.huji.ac.il}
\urladdr{http://www.math.rutgers.edu/$\sim$shelah}

\thanks{Both authors acknowledge support from the United States-Israel
Binational Science Foundation (Grant no. 2010405). This is publication
1022 of the second author.}

\date{May 31, 2013}

\begin{abstract}
  We show the consistency of ``there is a nice $\sigma$--ideal $\cI$
  on the reals with $\add(\cI)=\aleph_1$ which cannot be represented
  as the union of a strictly increasing sequence of length $\omega_1$
  of $\sigma$-subideals''. This answers Borodulin--Nadzieja and
  G{\l}\c{a}b \cite[Problem 6.2(ii)]{BnG11}.
\end{abstract}

\maketitle

\section{Introduction}
Borodulin--Nadzieja and G{\l}\c{a}b \cite{BnG11} studied generalizations
of the Mokobodzki ideal and they showed that those $\sigma$--ideals do not
have Borel bases of bounded Borel complexity. In \cite[Section 5]{BnG11}
they noticed that the unbounded Borel complexity of bases implies that the
additivity of the $\sigma$--ideal under consideration is $\aleph_1$. This
exposed the heart of Cicho\'{n} and Pawlikowski  \cite[Corollary
2.4]{CiPa86} and showed the influence of the existence of a strictly
increasing $\omega_1$--sequence of $\sigma$--subideals which add up to the
whole ideal.

Motivated by this, Borodulin--Nadzieja and G{\l}\c{a}b
introduced a new cardinal invariant $\cofin(\cI)$ associated with non-trivial
$\sigma$--ideals $\cI$: the minimal length of a strictly increasing sequence
of $\sigma$--subideals with union $\cI$ (see Definition
\ref{coefficients}). They showed that the additivity of the $\sigma$--ideal
$\cI$ is not larger than $\cofin(\cI)$ (see \cite[Proposition 5.2]{BnG11} or
Theorem \ref{BGclassic} here) and in \cite[Problem 6.2(ii)]{BnG11} they
asked if the two invariants can be different. In the present paper we answer
this question in positive.

In the second section we define the relevant cardinal invariants and we point
out situations when $\cofin(\cI)<\cof(\cI)$ for the meager and the null
ideals. In Section 3 we introduce a nicely definable $\sigma$--ideal $\cI_f$
with a Borel bases consisting of $\Pi^0_2$ sets. Then we show that,
consistently, $\add(\cI_f)=\aleph_1$ while $\cofin(\cI_f)=\aleph_2$
(Corollary \ref{mcor}).
\medskip

\noindent{\bf Notation}\qquad Most of our notation is standard and
compatible with that of classical textbooks (like Bartoszy\'nski and Judah
\cite{BaJu95}). However, in forcing we keep the older convention that {\em a
  stronger condition is the larger one}.  \medskip

$\bullet$ Ordinals will be denoted with initial letters of the Greek
alphabet ($\alpha$--$\zeta$) and integers (finite ordinals) will be
denoted by $i,j,k,\ell,m,n$. Letters $\kappa,\lambda,\mu$ will denote
uncountable cardinals.

$\bullet$ By a {\em sequence\/} we mean a function whose domain is a set of
  ordinals. Sequences will be denoted by letters
  $\eta,\nu,\rho,\sigma,\varsigma,\varphi,\psi$ (with possible indices) .

For two sequences $\eta,\nu$ we write $\nu\vtl\eta$ whenever $\nu$
is a proper initial segment of $\eta$, and $\nu \trianglelefteq\eta$
when either $\nu\vtl\eta$ or $\nu=\eta$.  The length of a sequence
$\eta$ is the order type of its domain and it is denoted by
$\lh(\eta)$.

$\bullet$ The power set of a set $X$ is denoted by $\cP(X)$ and the
collection of all subsets of $X$ of size $m$ is called $[X]^{\textstyle m}$
and the collection of all finite subsets of $X$ is denoted by
$[X]^{\textstyle {<}\aleph_0}$.

$\bullet$ The Cantor space $\can$ is the space of all functions from
$\omega$ to $2$, equipped with the product topology generated by sets of the
form $\{\eta\in\can:\nu\vartriangleleft\eta\}$ for $\nu\in\fs$.

$\bullet$ A family $\cI$ of subsets of $X$ which is closed under
finite unions and taking subsets is called an {\em ideal on $X$}. It
is a proper ideal if $X\notin \cI$ (i.e., $\cI\neq \cP(X)$) and it is
a $\sigma$--ideal if it is closed under countable unions.  The
$\sigma$--ideal of meager subsets of the Cantor space $\can$ is called $\cM$
and the $\sigma$--ideal of Lebesgue null sets is $\cN$.

$\bullet$ For a forcing notion $\bP$, all $\bP$--names for objects in the
extension via $\bP$ will be denoted with a tilde below (e.g.~$\name{A}$,
$\name{\eta}$). The canonical name for a $\bbP$--generic filter over $\bV$
is denoted $\name{G}_\bbP$. The Cohen forcing for adding $\kappa$ many Cohen
reals in $\can$ is called $\cohenk$ (so a condition in $\cohenk$ is a finite
function $p:\dom(p)\longrightarrow 2$ with $\dom(p)\subseteq \kappa\times
\omega$ and the order of $\cohenk$ is the inclusion). The forcing $\bbC$ is
$\bbC_1$.

\section{cofin and $\cM$, $\cN$}

\begin{definition}
\label{coefficients}
Let $\cI$ be an ideal on $X$. We define the following cardinal
characteristics of $\cI$:
\begin{enumerate}
\item $\add(\cI)=\min\{|\cA|:\cA\subseteq \cI\ \&\
  \bigcup\cA\notin \cI\}$;
\item $\cof(\cI)=\min\{|\cB|:\cB\subseteq \cI\ \&\
  (\forall A\in \cI)(\exists B\in\cB)(A\subseteq B)\}$;
\item $\cofin(\cI)$ is the minimal limit ordinal $\gamma$ for which
  there exists a sequence $\bar{\cI}=\langle\cI_\alpha:\alpha<
\gamma\rangle$ such that
  \begin{enumerate}
\item[(a)] $\cI=\bigcup\limits_{\alpha<\gamma} \cI_\alpha$ and
\item[(b)] $\cI_\alpha\subsetneq \cI_\beta$ for $\alpha<\beta<\gamma$,
  and
\item[(c)] each $\cI_\alpha$ is a $\sigma$--ideal,
\end{enumerate}
(or $\infty$ if there is no sequence $\bar{\cI}$ as above);
\item $\cofin^-(\cI)$ and $\cofin^*(\cI)$ are defined similarly to
  $\cofin(\cI)$, but clause (c) is replaced by (c)$^-$ and
  (c)$^*$, respectively, where
  \begin{enumerate}
\item[(c)$^-$] each $\cI_\alpha$ is an ideal;
\item[(c)$^*$] each $\cI_\alpha$ is closed under taking subsets (i.e.,
  $B\subseteq A\in\cI_\alpha$ implies $B\in\cI_\alpha$);
\end{enumerate}
\item $\cofin^+(\cI)$ is the minimal limit ordinal $\gamma$ for
  which there exists a sequence $\langle\cI_\alpha:\alpha<\gamma\rangle$
  such that  clauses (a),(b) and (c) of (3) above are satisfied and
  \begin{enumerate}
\item[(d)] all singletons belong to $\cI_0$.
\end{enumerate}
\end{enumerate}
\end{definition}

If $\cI$ is a non-principal ideal (i.e., $\cof(\cI)\geq\aleph_0$), then
$\cofin^-(\cI)$ is well defined and $\cofin^-(\cI)\leq\cof(\cI)$. To see
this, pick a basis $\{B_\zeta:\zeta<\cof(\cI)\}\subseteq \cI$ for $\cI$. Let
$\zeta_0$ be the first ordinal $\zeta\leq \cof(\cI)$ such that for some set
$B\in\cI$ every member of $\cI$ can be covered by finitely many elements
of $\{B_\vare:\vare<\zeta\}\cup\{B\}$. Necessarily, $\zeta_0$ is a limit
ordinal. Let $B^*\in\cI$ be such that $\{B_\vare:\vare<\zeta_0\}\cup\{B^*\}$
generates $\cI$, i.e., every set in $\cI$ can be covered by $B^*$ and
finitely many sets $B_\vare$ with $\vare<\zeta_0$. For $\zeta<\zeta_0$ let
$\cI_\zeta$ be the ideal generated by $\{B_\vare:\vare<\zeta\}\cup
\{B^*\}$. Then $\cI=\bigcup\limits_{\zeta<\zeta_0}\cI_\zeta$ and, by the
minimality of $\zeta_0$, the sequence $\langle\cI_\zeta:\zeta<\zeta_0
\rangle$ does not stabilize. Consequently, we may choose an increasing
sequence $\langle \zeta_\alpha:\alpha<\cf(\zeta_0)\rangle$ cofinal in
$\zeta_0$ and such that $\langle\cI_{\zeta_\alpha}:\alpha<\cf(\zeta_0)
\rangle$ is a strictly increasing sequence of ideals with the union $\cI$.

Replacing above ``ideal'' with ``$\sigma$-ideal'' and ``finitely many'' with
``countably many'' we will get an argument showing that $\cofin(\cI)$ is
well defined for a $\sigma$--ideal $\cI$ and $\cofin(\cI)\leq\cof(\cI)$.

The cardinal invariant $\cofin$ was introduced by Borodulin--Nadzieja
and G{\l}\c{a}b in \cite[Section 5]{BnG11}. They showed that, for a
non--trivial $\sigma$--ideal $\cI$, $\cofin(\cI)$ is a well defined regular
cardinal and that the following inequalities are satisfied.

\begin{theorem}
[Borodulin--Nadzieja and G{\l}\c{a}b {\cite[Section 5]{BnG11}}]
\label{BGclassic}
Let $\cI$ be a non-principal $\sigma$-ideal of subsets of $X$. Then
\[\add(\cI)\leq \cofin^*(\cI)\leq \cofin^-(\cI)\leq \cofin(\cI)\leq
\cof(\cI).\]
\end{theorem}

\begin{proposition}
\label{CohRan}
Let $\kappa=\kappa^{\aleph_0}$ be an uncountable cardinal.
  \begin{enumerate}
  \item The Cohen algebra $\bbC_\kappa$ for adding $\kappa$ many Cohen
    reals forces that
\[\add(\cM)=\cofin(\cM)=\cofin^+(\cM)=\aleph_1\leq
\cof(\cM)= \kappa=2^{\aleph_0}.\]
\item The Solovay algebra $\bbB_\kappa$ for adding $\kappa$ many
  random  reals forces that
\[\add(\cN)=\cofin(\cN)=\cofin^+(\cN)=\aleph_1\leq
\cof(\cN)= \kappa=2^{\aleph_0}.\]
  \end{enumerate}
\end{proposition}

\begin{proof}
(2)\quad In both cases the proof is essentially the same, so let us argue
for the Solovay algebra only. Represent $\kappa$ as the disjoint union
$\kappa=\bigcup\limits_{\vare<\omega_1} K_\vare$ where each $K_\vare$ is of
size $\kappa$. For $\vare<\omega_1$ set $\alpha_\vare=\min(K_\vare)$ and
$A_\vare=\bigcup\limits_{\zeta<\vare} K_\zeta$.

Suppose that $\bar{r}=\langle r_\alpha:\alpha<\kappa\rangle$ is a
$\bbB_\kappa$--generic over $\bV$, so $r_\alpha\in\can$ are random
reals, and let us argue in $\bV[\bar{r}]$. For each $\vare<\omega_1$ let
$\cI_\vare$ be the $\sigma$--ideal generated by singletons and the
family of all Borel null sets coded in $\bV[r_\alpha:\alpha\in A_\vare]$.
Then $\langle \cI_\vare:\vare<\omega_1\rangle$ is an increasing sequence of
$\sigma$--ideals, $\cI_0$ contains all singletons and $\cN=
\bigcup\limits_{\vare<\omega_1} \cI_\vare$. Moreover,  for each
$\vare<\omega_1$,
\[B\stackrel{\rm def}{=}\big\{x\in \can:(\forall n<\omega)(x(2n)=
r_{\alpha_\vare}(2n))\big\} \in \cI_{\vare+1}\setminus \cI_\vare.\]
Why? Clearly, $B$ is a Borel null set coded in $\bV[r_\alpha:\alpha\in
A_{\vare+1}]$, so $B\in \cI_{\vare+1}$. Suppose $B_i$ are Borel null sets
coded in $\bV[r_\alpha:\alpha\in A_\vare]$ and $x_i\in(\can)^{\bV[\bar{r}]}$
(for $i<\omega$). Choose $x^*\in\{x\in\can\cap\bV: (\forall
n<\omega)(x(2n)=0)\}\setminus\{x_i+r_{\alpha_\vare}:
i<\omega\}$. Then $x^*+r_{\alpha_\vare}$ is a random real over
$\bV[r_\alpha:\alpha\in A_\vare]$, so $x^*+r_{\alpha_\vare}\in B\setminus
(\bigcup\limits_{i<\omega} B_i\cup\{x_i:i<\omega\})$. Thus we may conclude
that $B\notin \cI_\vare$.
\end{proof}

\begin{definition}
[Ros{\l}anowski and Shelah {\cite[Definition 3.4]{RoSh:972}}]
\label{klbase}
Let $\cI$ be an ideal of subsets of a space $X$ and $\alpha^*,\beta^*$
be limit ordinals. {\em An $\alpha^*\times\beta^*$--base for $\cI$} is
an indexed family $\cB=\{ B_{\alpha,\beta}:\alpha<\alpha^*\ \&\ \beta <
\beta^*\}$ of sets  from $\cI$ such that
\begin{enumerate}
\item[(i)] $\cB$ is a basis for $\cI$, i.e., $(\forall A\in\cI)(\exists B\in
  \cB)(A\subseteq B)$, and
\item[(ii)] for each $\alpha_0,\alpha_1<\alpha^*$,
  $\beta_0,\beta_1<\beta^*$ we have
\[B_{\alpha_0,\beta_0}\subseteq B_{\alpha_1,\beta_1}\quad \Leftrightarrow\quad
\alpha_0\leq \alpha_1\ \&\ \beta_0\leq\beta_1.\]
\end{enumerate}
\end{definition}

If follows from results of Bartoszy\'nski and Kada \cite{BaKa05} (for
the meager ideal) and Burke and Kada \cite{BuKa04} (for the null
ideal) that for any cardinals $\kappa$ and $\lambda$ of uncountable
cofinality we may force that $\cM$ has a $\kappa\times\lambda$--basis,
and we may also force that $\cN$ has a $\kappa\times
\lambda$--basis. In \cite[Theorem 3.7]{RoSh:972} we constructed a
model in which both ideals have $\kappa\times\lambda$--bases.

\begin{proposition}
\label{basegivescofin}
Let $\kappa,\lambda$ be regular uncountable cardinals, $\kappa\leq
\lambda$,
\begin{enumerate}
\item If $\cI$ is a $\sigma$--ideal on a space $X$ and $\cI$ has a
  $\kappa\times\lambda$--base, then
\[\kappa=\add(\cI)=\cofin(\cI)\quad \mbox{ and }\quad \cof(\cI)=
\lambda.\]
\item There is a ccc forcing notion $\bbP$ forcing that $2^{\aleph_0}
=\lambda^{\aleph_0}$ and
\begin{enumerate}
\item[(i)] the $\sigma$--ideal $\cN$ has a $\kappa\times
\lambda$--base $\{ A_{\alpha,\beta}:\alpha<\kappa,\ \beta<\lambda\}$ with
the property that
\[\alpha_0>\alpha_1\ \vee\ \beta_0>\beta_1\quad \Rightarrow\quad
|A_{\alpha_0,\beta_0}\setminus A_{\alpha_1,\beta_1}|=2^{\aleph_0},\]
and
\item[(ii)] the $\sigma$--ideal $\cM$ has a $\kappa\times
\lambda$--base $\{ B_{\alpha,\beta}:\alpha<\kappa,\ \beta<\lambda\}$ with
the property that
\[\alpha_0>\alpha_1\ \vee\ \beta_0>\beta_1\quad \Rightarrow\quad
|B_{\alpha_0,\beta_0}\setminus B_{\alpha_1,\beta_1}|=2^{\aleph_0}.\]
\end{enumerate}
In particular,
\[\begin{array}{ll}
\forces_{\bbP} \mbox{``}&
\add(\cM)=\add(\cN)= \cofin^+(\cM)=\cofin^+(\cN)=\kappa\ \mbox{ and}\\
&\cof(\cM)=\cof(\cN)=\lambda.\mbox{ ''}\end{array}\]
\end{enumerate}
\end{proposition}

\begin{proof}
(1)\quad   Assume that  $\{B_{\alpha,\beta}:\alpha< \kappa,\ \beta<
\kappa\}$  is a $\kappa\times\lambda$--base for $\cI$. It should be clear
that then $\kappa=\add(\cI)$ and $\cof(\cI)=\lambda$.

Let us argue that $\cofin(\cI)\leq \kappa$. For $\zeta<\kappa$ let
$\cI_\zeta$ be the $\sigma$--ideal generated by the family
$\{B_{\alpha,\beta}:\alpha \leq \zeta\ \&\ \beta<\lambda\}$. Plainly,
$\langle \cI_\zeta:\zeta<\kappa\rangle$ is an
increasing sequence of $\sigma$--ideals such that $\cI=
\bigcup\limits_{\zeta<\kappa} \cI_\zeta$. We claim that
$B_{\zeta+1,0}\in \cI_{\zeta+1}\setminus \cI_\zeta$. Suppose that
$I\subseteq (\zeta+1)\times\lambda$ is countable. Then we may choose
$\beta^*<\lambda$ such that $I\subseteq (\zeta+1)\times\beta^*$ and
consequently $\bigcup\{B_{\alpha,\beta}: (\alpha,\beta)\in I\}\subseteq
B_{\zeta,\beta^*}$. But $B_{\zeta+1,0}\nsubseteq B_{\zeta,\beta^*}$ and so
$B_{\zeta+1,0}\nsubseteq \bigcup\{B_{\alpha,\beta}:(\alpha,\beta)\in
I\}$. Now we may conclude now that $B_{\zeta+1,0}\notin \cI_\zeta$.
\medskip

\noindent (2)\quad The forcing notion $\bbQ^{\kappa,\lambda}$
constructed in the proof of \cite[Theorem 3.7]{RoSh:972} has the
desired properties.
\end{proof}

\section{cofin and $\cI_f$}
We introduce here a nicely definable Borel ideal $\cI_f$ for which,
consistently, $\add(\cI_f)<\cofin(\cI_f)$. The proof of the consistency will
reseamble Shelah \cite[Chapter II, Theorem 4.6]{Sh:f} (and thus also
\cite{Sh:98}). The respective forcing notion is obtained by means of FS
iteration of ccc forcing notions, however the iteration itself is forced
too.

\begin{context}
In this section we fix two strictly increasing functions
$f,g:\omega\longrightarrow \omega$ such that for each $n<\omega$ we have
\[2<g(n)<f(n)\quad\mbox{ and }\quad \frac{g(n)}{f(n)}\leq
\frac{1}{n+1}.\]
\end{context}

\begin{definition}
\label{sladef}
\begin{enumerate}
\item {\em A null slalom below $f$} is a function
$\varphi\in\prod\limits_{n<\omega} \cP\big(f(n)\big)$ such that
$\lim\limits_{n\to\infty} \frac{|\varphi(n)|}{f(n)}=0$.
\item Let $\cS_f$ be the collection of all null slaloms below $f$ and
  let $\cX_f=\prod\limits_{n<\omega} f(n)$ be equipped with the
  natural product topology (so $\cX_f$ is a Polish space).
\item For $\varphi\in\cS_f$ we define
\[[\varphi]=\big\{x\in \cX_f:\big(\exists^\infty n<\omega \big)
\big(x(n)\in\varphi(n) \big) \big\}.\]
\end{enumerate}
\end{definition}

\begin{observation}
\label{basobs}
Let $\varphi_i\in\cS_f$ (for $i<\omega$).
\begin{enumerate}
\item $[\varphi_0]\subseteq [\varphi_1]$ if and only if
  $\big(\forall^\infty n<\omega \big) \big(\varphi_0(n) \subseteq
  \varphi_1(n) \big)$.
\item There is $\psi\in\cS_f$ such that $\bigcup\limits_{i<\omega}
[\varphi_i]\subseteq [\psi]$.
\end{enumerate}
\end{observation}

\begin{definition}
\label{theideal}
Let $\cI_f$ be the $\sigma$--ideal of subsets of $\cX_f$ generated by
all sets $[\varphi]$ for $\varphi\in \cS_f$. Thus, by Observation
\ref{basobs},
\[\cI_f=\big\{A\subseteq \cX_f:\big(\exists \varphi\in \cS_f \big)
\big(A\subseteq [\varphi]\big) \big\}.\]
\end{definition}

We will construct a forcing notion $\bbP$ forcing that
$\add(\cI_f)<\cofin(\cI_f)$, but first  we need several technical
ingredients.

\begin{definition}
\label{ourCohen}
For a cardinal $\kappa$ we define a forcing notion $\bbQ_0^\kappa$:

\noindent{\em A condition $p$ in $\bbQ_0^\kappa$} is a finite function such
that $\dom(p)\subseteq \kappa$ and for some $n=n^p<\omega$, for all
$\vare\in\dom(p)$ we have $p(\vare)\in \prod\limits_{i<n}
[f(i)]^{\textstyle g(i)}$.

\noindent{\em The order $\leq=\leq_{\bbQ_0^\kappa}$ of $\bbQ_0^\kappa$} is
defined by\\
$p\leq q$ if and only if ($p,q\in\bbQ_0^\kappa$ and) $\dom(p)\subseteq
\dom(q)$ and $(\forall \vare\in\dom(p))(p(\vare)\trianglelefteq
q(\vare))$.
\medskip

For $\vare<\kappa$, a $\bbQ_0^\kappa$--name $\name{\nu}(\vare)$ is
defined by
\[\forces_{\bbQ_0^\kappa}\name{\nu}(\vare)=\bigcup\big\{p(\vare):\vare \in
\dom(p)\  \&\ p\in\name{G}_{\bbQ_0^\kappa}\big\}.\]
\end{definition}

\begin{observation}
\label{obsCoh}
\begin{enumerate}
\item The forcing notion $\bbQ_0^\kappa$ is equivalent to $\bbC_\kappa$,
  the forcing adding $\kappa$ many Cohen reals.
\item $\forces_{\bbQ_0^\kappa}\mbox{`` for every $\vare<\kappa$ we have
  }\name{\nu}(\vare)\in\prod\limits_{i<\omega}   [f(i)]^{\textstyle
    g(i)}\subseteq\cS_f$ ''.
\end{enumerate}
\end{observation}

\begin{definition}
\label{forforsla}
Let $\mu$ be an infinite cardinal and $\bar{\varphi}=\langle
\varphi_\zeta:\zeta<\mu\rangle$ be a sequence of null slaloms below
$f$ (so $\varphi_\zeta\in\cS_f$ for $\zeta<\mu$). We define a
forcing notion $\qfor$:

\noindent{\em A condition in $\qfor$} is a tuple $p=(k^p,m^p,
u^p,\sigma^p)=(k,m,u,\sigma)$ such that
\begin{enumerate}
\item[(a)] $k,m<\omega$, $\emptyset\neq u\in [\mu]^{\textstyle
    {<}\aleph_0}$,  $\sigma \in\prod\limits_{i<k}\cP(f(i))$, and
\item[(b)] for each $\ell\geq k$ and $\zeta\in u$ we have
  $|\varphi_\zeta(\ell)| < \frac{f(\ell)}{m\cdot |u|}$.
\end{enumerate}

\noindent{\em The order $\leq=\leq_{\qfor}$ of $\qfor$} is defined
by\\
$p\leq q$ if and only if ($p,q\in\qfor$ and) $k^p\leq k^q$, $m^p\leq
m^q$, $u^p\subseteq u^q$, $\sigma^p\trianglelefteq \sigma^q$ and for
each $\ell\in [k^p,k^q)$ we have
\[|\sigma^q(\ell)|\leq \frac{f(\ell)}{m^p}\quad\mbox{ and }\quad
\bigcup\big\{\varphi_\zeta(\ell): \zeta\in u^p\big\}\subseteq
\sigma^q(\ell). \]

We also define a $\qfor$--name $\name{\varsigma}$ by
\[\forces_{\qfor}\name{\varsigma}=\bigcup\big\{\sigma^p: p\in
\name{G}_{\qfor} \big\}.\]
\end{definition}

\begin{proposition}
\label{forprop}
Let $\mu$ be an infinite cardinal and $\bar{\varphi}= \langle
\varphi_\zeta:\zeta<\mu\rangle\subseteq\cS_f$. Then $\qfor$ is a well
defined ccc forcing notion of size $\mu$ and
\[\forces_{\qfor}\mbox{`` } \name{\varsigma}\in\cS_f\ \&\
\bigcup_{\zeta<\mu} [\varphi_\zeta]\subseteq [\name{\varsigma}]\in
\cI_f\mbox{ ''.}\]
\end{proposition}

\begin{proof}
First note that if $p\in\qfor$ and $m=m^p$, $k=k^p+1$, $u=u^p$ and
$\sigma=\sigma^p\conc \langle \bigcup\limits_{\zeta\in u}
\varphi_\zeta(k^p)\rangle$, then $(k,m,u,\sigma)\in\qfor$ is a
condition stronger than $p$. Hence we may conclude that
$\forces_{\qfor} \name{\varsigma}\in\prod\limits_{i<\omega}\cP(f(i))$.

Also, if $p\in\qfor$ and $m>m^p$, then we may find $k>k^p$ such that
$|\varphi_\zeta(\ell)|<\frac{f(\ell)}{m\cdot |u^p|}$ for all $\zeta\in
u^p$ and $\ell\geq k$. Let $u=u^p$ and $\sigma\in \prod\limits_{i<k}
\cP(f(i))$ be such that $\sigma(\ell)=\sigma^p(\ell)$ for $\ell<k^p$
and $\sigma(\ell)=\bigcup\limits_{\zeta\in u} \varphi_\zeta(\ell)$
for $\ell\in [k^p,k)$. Then $(k,m,u,\sigma)\in\qfor$ is a condition
stronger than $p$ and it forces that $|\name{\varsigma}(\ell)|\leq
\frac{f(\ell)}{m}$ for all $\ell\geq k$. Hence we may conclude that
$\forces_{\qfor} \name{\varsigma}\in\cS_f$.

It follows from the definition of the order of $\qfor$ that
\[p\forces_{\qfor} \big(\forall\ell\geq k^p\big)\big(\forall\zeta\in
u^p\big)\big(\varphi_\zeta(\ell)\subseteq
\name{\varsigma}(\ell)\big),\]
and hence easily $\forces_{\qfor}\bigcup\limits_{\zeta<\mu}
[\varphi_\zeta] \subseteq [\name{\varsigma}]$.
\medskip

Let us argue now that  $\qfor$ satisfies the ccc. Suppose $\langle p_\vare:
\vare< \omega_1\rangle \subseteq \qfor$.  For each $\vare<\omega_1$, we  may
find $K^\vare>k^{p_\vare}$ such that
\begin{enumerate}
\item[$(\oplus)_1$] $\big (\forall \ell\geq K^\vare\big)
  \big(\forall\zeta\in u^{p_\vare} \big) \big(|\varphi_\zeta(\ell)|<
  \frac{f(\ell)}{2\cdot |u^{p_\vare}|\cdot m^{p_\vare}}\big)$
\end{enumerate}
and define $\rho^\vare\in \prod\limits_{i<K^\vare}\cP(f(i))$ so that
$\rho^\vare(\ell) =\sigma^{p_\vare}(\ell)$ for $\ell<k^{p_\vare}$
and $\rho^\vare(\ell)=\bigcup\limits_{\zeta\in u^{p_\vare}}
\varphi_\zeta(\ell)$ for $\ell\in [k^{p_\vare},K^\vare)$. Then we may find
an uncountable set $S\subseteq \omega_1$ and $K^*, m^*, \rho^*,\ell^*$ such
that for all $\vare\in S$:
\begin{enumerate}
\item[$(\oplus)_2$] $K^*=K^\vare$, $m^*=m^{p_\vare}$,  $\rho^*=\rho^\vare$
  and  $|u^\vare|=\ell^*$.
\end{enumerate}
Consider distinct $\vare_0,\vare_1\in S$: letting $u^*=u^{\vare_0}\cup
u^{\vare_1}$ we get a condition $(K^*,m^*,u^*,\rho^*)\in\qfor$ stronger
than both $p_{\vare_0}$ and $p_{\vare_1}$.
\end{proof}

\begin{definition}
\label{Ydef}
Let $\kappa<\lambda$ be uncountable regular cardinals.
\begin{enumerate}
\item A {\em Y--iteration for $\kappa,\lambda$} is a finite support
  iteration $\langle\bbP_\beta, \name{\bbQ}_\beta: \beta<\alpha\rangle$  of
  ccc forcing notions such that the following demands
  $(\otimes)_1$--$(\otimes)_3$ are satisfied.
\begin{enumerate}
\item[$(\otimes)_1$] $0<\alpha\leq \lambda$ and
  $\name{\bbQ}_0=\bbQ_0^\kappa$ is the forcing notion adding $\kappa$ Cohen
  reals as represented in Definition \ref{ourCohen} with
  $\bbQ_0^\kappa$--names $\name{\nu}(\vare)$ (for  $\vare<\kappa$) as
  defined there.
\item[$(\otimes)_2$] For each $\beta<\alpha$ we have
  $\forces_{\bbP_\beta} |\name{\bbQ}_\beta|\leq\lambda$.
\item[$(\otimes)_3$] Let $n<\omega$. Suppose that $\langle p_\zeta:
  \zeta<\kappa\rangle\subseteq\bbP_\alpha$ and  $\langle\delta_\zeta:
  \zeta<\kappa\rangle \subseteq \kappa$ and $\delta_\zeta\neq
  \delta_{\zeta'}$ for $\zeta<\zeta'<\kappa$. Then there are
  $q\in\bbP_\alpha$, $m>n$, $v\subseteq\kappa$, and $A_\zeta$ (for
  $\zeta\in v$) such that
\begin{enumerate}
\item[(i)] $|v|\geq\frac{f(m)}{2\cdot g(m)}$,
\item[(ii)] $p_\zeta\leq q$ for all $\zeta\in v$,
\item[(iii)] $A_\zeta\in [f(m)]^{\textstyle g(m)}$ (for $\zeta\in
  v$) are pairwise disjoint sets,
\item[(iv)] $q\forces_{\bbP_\alpha}\mbox{`` }\big(\forall \zeta\in
  v\big)\big(\name{\nu}(\delta_\zeta)(m)= A_\zeta\big)\mbox{ ''}$.
\end{enumerate}
\end{enumerate}
\item The collection of  all Y--iterations for $\kappa,\lambda$ of length
  ${<}\lambda$ which belong to $\cH(\beth_\lambda^+)$  is denoted by
  $\bbY^\lambda_\kappa$. It is ordered by the end-extension of sequences
  $\trianglelefteq$.
\end{enumerate}
\end{definition}

The condition \ref{Ydef}(1)$(\otimes)_3$ implies that the null slaloms added
at the first step of a Y--iteration provide a family of sets whose union is
not included in any null slalom. Note that in \ref{Ydef}$(\otimes)_3$
necessarily $|v|\leq \frac{f(m)}{g(m)}$.

\begin{lemma}
\label{cruciallemma}
Assume $\kappa<\lambda$ are regular uncountable cardinals. Suppose that
$\langle\bbP_\beta,\name{\bbQ}_\beta: \beta<\alpha\rangle$ is a
Y--iteration for $\kappa,\lambda$. Then $\forces_{\bbP_\alpha} \add(\cI_f)
\leq\kappa$.
\end{lemma}

\begin{proof}
We know that for each $\vare<\kappa$ we have $\forces_{\bbP_\alpha}
\name{\nu}(\vare)\in \cS_f$ (remember Observation \ref{obsCoh}) and
therefore $\forces_{\bbP_\alpha}\{[\name{\nu}(\vare)]:\vare<\kappa\}
\subseteq \cI_f$. We are going to argue that
\[\forces_{\bbP_\alpha}\bigcup\big\{[\name{\nu}(\vare)]:\vare<\kappa
\big\} \notin \cI_f.\]
Suppose towards contradiction that this is not the case. Then we may
pick $p\in\bbP_\alpha$ and a $\bbP_\alpha$--name $\name{\varphi}$ such that
\[p\forces_{\bbP_\alpha}\name{\varphi}\in\cS_f\ \&\ \big(\forall
\vare<\kappa\big) \big(\forall^\infty n<\omega\big)\big( \name{\nu}
(\vare)(n)\subseteq\name{\varphi}(n)\big)\]
(remember Observation \ref{basobs}). Now for each $\vare<\kappa$
we pick a condition $p_\vare\geq p$ and an integer $n_\vare<\omega$
such that
\[p_\vare\forces_{\bbP_\alpha}\big(\forall n\geq n_\vare\big)\big(
\name{\nu} (\vare)(n)\subseteq\name{\varphi}(n)\ \&\
\frac{|\name{\varphi}(n)|}{f(n)} <1/4\big).\]
For some $n^*<\omega$ the set $S=\{\vare<\kappa: n_\vare=n^*\}$ is
of size $\kappa$. Apply \ref{Ydef}(1)$(\otimes)_3$ to $\langle p_\vare:
\vare\in S\rangle\subseteq\bbP_\alpha$ and $\langle \vare:\vare\in S\rangle
\subseteq\kappa$ and $n=n^*$ to find $q\in\bbP_\alpha$, $m>n^*$, $v\subseteq
S$, and $A_\vare$ (for $\vare\in v$) such that conditions (i)--(iv) there
hold. Then
\[q\forces_{\bbP_\alpha}\mbox{`` }\bigcup_{\vare\in v} A_\vare=
\bigcup_{\vare\in v} \name{\nu}(\vare)(m)\subseteq \name{\varphi}(m)\ \&\
|\name{\varphi}(m)|<\frac{f(m)}{4}\mbox{ ''.}\]
But $|\bigcup\limits_{\vare\in v} A_\vare|= |v|\cdot g(m)\geq
\frac{f(m)}{2}$, a contradiction.
\end{proof}

\begin{context}
For the rest of this section we fix uncountable regular cardinals
$\kappa<\lambda$ such that $\lambda^\kappa=\lambda$. Also, instead of
``Y--iteration for $\kappa,\lambda$'' we will just say ``Y--iteration''.
\end{context}

\begin{lemma}
\label{manyYiter}
\begin{enumerate}
\item $\langle \bbP_0,\bbQ_0^\kappa\rangle$ is a Y--iteration (of length
  1).
\item Asssume that $\langle \bbP_\beta,\name{\bbQ}_\beta:\beta<\alpha
\rangle$ is a Y--iteration of length $\alpha<\lambda$ and $\name{\bbQ}$ is a
$\bbP_\alpha$--name for a ccc forcing notion of size ${<}\kappa$ (i.e.,
$\forces_{\bbP_\alpha}|\name{\bbQ}|<\kappa$). Then $\langle
  \bbP_\beta,\name{\bbQ}_\beta:\beta<\alpha\rangle\conc\langle\bbP_\alpha,
  \bbQ \rangle$ is a Y--iteration of length $\alpha+1$. In particular,
$\langle\bbP_\beta,\name{\bbQ}_\beta:\beta<\alpha\rangle\conc\langle
\bbP_\alpha,\bbC \rangle$ is a Y--iteration.
\item If $\langle \bbP_\beta,\name{\bbQ}_\beta:\beta<\alpha
\rangle$ is a Y--iteration and $\name{\bbQ}$ is a $\bbP_\alpha$--name for a
$\sigma$--centered forcing, then $\langle\bbP_\beta,\name{\bbQ}_\beta:
\beta<\alpha\rangle\conc\langle\bbP_\alpha,\bbQ \rangle$ is a Y--iteration.
\item If $\gamma\leq\lambda$ is a limit ordinal and $\langle \bbP_\beta,
  \name{\bbQ}_\beta:\beta<\gamma\rangle$ is an FS iteration such that
  $\langle \bbP_\beta,\name{\bbQ}_\beta:\beta<\alpha\rangle$ is a
  Y--iteration for every $\alpha<\gamma$, then $\langle \bbP_\beta,
  \name{\bbQ}_\beta:\beta<\gamma\rangle$ is a Y--iteration.
\item $(\bbY^\lambda_\kappa,\trianglelefteq)$ is a ${<}\lambda$--complete
  forcing notion.
\end{enumerate}
\end{lemma}

\begin{proof}
In all cases the only demand of \ref{Ydef}(1) that needs to be verified is
$(\otimes)_3$.
\medskip

\noindent (1)\quad Let $\bbQ^\kappa_0$ be the forcing notion adding $\kappa$
Cohen reals as described in Definition \ref{ourCohen}.  Let $n<\omega$,
$\delta_\zeta\in\kappa$ and $p_\zeta\in\bbQ_0^\kappa$ (for $\zeta<\kappa$)
satisfy the assumptions of \ref{Ydef}(1)$(\otimes)_3$.  By making
conditions $p_\zeta$ stronger and possibly passing to a subsequence, we
may assume also that:
\begin{enumerate}
\item[$(*)_1$] $\delta_\zeta\in\dom(p_\zeta)$ for all $\zeta<\kappa$,
\item[$(*)_2$] for some $m>n+2$, for all $\zeta<\kappa$, we have
  $n^{p_\zeta}=m$ (so $p_\zeta(\vare)\in \prod\limits_{i<m}
  [f(i)]^{\textstyle g(i)}$ for $\vare\in\dom(p_\zeta)$),
\item[$(*)_3$] the family $\{\dom(p_\zeta):\zeta<\kappa\}$ forms a
  $\Delta$--system of finite sets and for all $\zeta,\zeta'<\kappa$ the
  conditions $p_\zeta,p_{\zeta'}$ are compatible.
\end{enumerate}
Pick any $v\subseteq\kappa$ of size $\lceil \frac{f(m)}{2\cdot
  g(m)}\rceil$. Since
\[\lceil \frac{f(m)}{2\cdot g(m)}\rceil\cdot g(m)\leq \frac{f(m)}{2}+
g(m)\leq\frac{f(m)}{2}+\frac{f(m)}{m+1}< f(m),\]
we may choose pairwise disjoint sets $A_\zeta\in [f(m)]^{\textstyle g(m)}$
(for $\zeta\in v$). Now define a condition $q\in\bbQ_0^\kappa$ so that
$\dom(q)=\bigcup\{\dom(p_\zeta): \zeta\in v\}$, $n^q=m+1$  and for
$\vare\in \dom(p_\zeta)$ the sequence $q(\vare)$ extends $p_\zeta(\vare)$
and $q(\delta_\zeta)(m)=A_\zeta$ (for $\zeta\in v$).
\medskip

\noindent (2)\quad Without loss of generality, for some ordinal $\gamma^*
<\kappa$ we have $\forces_{\bbP_\alpha}$`` the set of conditions in
$\name{\bbQ}$ is $\gamma^*$ ''. Let $n<\omega$ and $p_\zeta\in\bbP_{\alpha+1},
\delta_\zeta\in\kappa$ (for $\zeta<\kappa$) satisfy the assumptions of
\ref{Ydef}(1)$(\otimes)_3$.  We may make our conditions stronger and we may
pass to a subsequence, so we may assume that $\alpha\in \dom(p_\zeta)$ and
$p_\zeta(\alpha)=\gamma<\gamma^*$ is an actual object, not a name (for
$\zeta<\kappa$).  Apply the assumption of \ref{Ydef}(1)$(\otimes)_3$ for
$\langle\bbP_\beta,\name{\bbQ}_\beta: \beta<\alpha\rangle$ to
$n,p_\zeta\rest\alpha,\delta_\zeta$ (for $\zeta<\kappa$) and choose $m>n$,
$q^*\in \bbP_\alpha$, $v\subseteq \kappa$ and pairwise disjoint sets
$A_\zeta\subseteq f(m)$ each of size $g(m)$ (for $\zeta\in v$) such that
\begin{itemize}
\item $|v|\geq \frac{f(m)}{2\cdot g(m)}$ and
\item $q^*$ is stronger than all $p_\zeta\rest\alpha$ for $\zeta\in v$ and it
  forces that $\name{\nu}(\delta_\zeta)(m)=A_\zeta$ (for $\zeta\in v$).
\end{itemize}
Let $q\in \bbP_{\alpha+1}$ be such that $q\rest \alpha=q^*$ and
$q(\alpha)=\gamma$.
\medskip

\noindent (3)\quad Assume that $\forces_{\bbP_\alpha}$`` $\name{\bbQ}$
is a $\sigma$--centered forcing notion '' and fix a
$\bbP_\alpha$--name $\name{F}$ such that
\[\begin{array}{ll}
\forces_\bbP&\mbox{`` }\name{F}:\name{\bbQ}\longrightarrow\omega
\mbox{ is a function satisfying:}\\
&\ \ \mbox{ if }x_0,\ldots,x_k\in\name{\bbQ}, \ k<\omega,\
\mbox{ and }\name{F}(x_0)=\ldots=\name{F}(x_k)=m,\\
&\ \mbox{ then the conditions } x_0,\ldots,x_k\mbox{ have a common
upper bound in }\name{\bbQ}\mbox{ ''}.
\end{array}\]
Suppose that $n<\omega$ and $p_\zeta\in\bbP_{\alpha+1},
\delta_\zeta\in\kappa$ (for $\zeta<\kappa$) satisfy the assumptions of
\ref{Ydef}(1)$(\otimes)_3$. By making the conditions stronger and passing
to a subsequence we may demand that $\alpha\in \dom(p_\zeta)$ and for
some $M<\omega$ we also have $p_\zeta\rest\alpha\forces_{\bbP_\alpha}$``
$\name{F}(p_\zeta(\alpha))=M$ ''. Use the assumption of
\ref{Ydef}(1)$(\otimes)_3$ for $\langle\bbP_\beta,\name{\bbQ}_\beta:
\beta<\alpha\rangle$ for $n,p_\zeta\rest\alpha,\delta_\zeta$
(for $\zeta<\kappa$) to find $m>n$, $q^*\in \bbP_\alpha$, $v\subseteq
\kappa$ and pairwise disjoint sets $A_\zeta\in [f(m)]^{\textstyle g(m)}$
(for $\zeta\in v$) such that
\begin{itemize}
\item $|v|\geq \frac{f(m)}{2\cdot g(m)}$ and
\item $q^*$ is stronger than all $p_\zeta\rest\alpha$ for $\zeta\in v$ and it
  forces that $\name{\nu}(\delta_\zeta)(m)=A_\zeta$ (for $\zeta\in v$).
\end{itemize}
Then also the condition $q^*$ forces that $\name{F}(p_\zeta(\alpha))=
M$ for all $\zeta\in v$, and thus we may pick a $\bbP_\alpha$--name
$\name{q}_\alpha$ such that $q^*\forces$ `` $\name{q}_\alpha$ is a condition
stronger than all $p_\zeta(\alpha)$ for $\zeta\in v$ ''. Define
$q\in \bbP_{\alpha+1}$ by $q\rest \alpha=q^*$ and $q(\alpha)=\name{q}_\alpha$.
\medskip

\noindent (4)\quad Let $n,p_\zeta,\delta_\zeta$ (for $\zeta<\kappa$) be as in
the assumptions of \ref{Ydef}(1)$(\otimes)_3$. By passing to a subsequence
we may also demand that $\{\dom(p_\zeta):\zeta<\kappa\}$ is a
$\Delta$--system of finite subsets of $\gamma$ with root $D$. Pick
$\alpha<\gamma$ such that $D\subseteq \alpha$. Since $\langle
\bbP_\beta,\name{\bbQ}_\beta: \beta<\alpha\rangle$ is a Y--iteration, we may
apply \ref{Ydef}(1)$(\otimes)_3$ to $n,\delta_\zeta$ and
$p_\zeta\rest\alpha$ (for $\zeta<\kappa$). This will give us
$q^*,v$ and $A_\zeta$ (for $\zeta\in v$) satisfying (i)--(iv) there (with
$p_\zeta\rest \alpha$ in place of $p_\zeta$ and $q^*$ in place of $q$). Let
$q\in\bbP_\gamma$ be such that $\dom(q)=\dom(q^*)\cup\bigcup\{\dom
(p_\zeta):\zeta\in v\}$ and $q\rest\alpha=q^*$ and $q(\beta)=p_\zeta(\beta)$
whenever $\zeta\in v$, $\beta\in \dom(p_\zeta)\setminus\alpha$.
\medskip

\noindent (5)\quad Follows from (3).
\end{proof}

\begin{lemma}
\label{onemore}
Assume that
\begin{enumerate}
\item[(a)] $\aleph_0\leq\mu\leq \kappa$ is a regular cardinal,
  $\alpha<\lambda$ is a limit ordinal of cofinality $\mu$ and  $\langle
  \alpha(\zeta):\zeta< \mu\rangle$ is a strictly increasing sequence cofinal
  in $\alpha$,
\item[(b)] $\langle\bbP_\beta,\name{\bbQ}_\beta:\beta< \alpha \rangle$
  is a Y--iteration,
\item[(c)] $\name{\bar{\varphi}}= \langle \name{\varphi}_\zeta: \zeta< \mu
  \rangle$ is a $\bbP_{\alpha(0)}$--name for a $\mu$--sequence of null
  slaloms below $f$ (so $\forces \name{\varphi}_\zeta \in\cS_f$),
\item[(d)]  for each $\zeta<\mu$ we have $\forces_{\bbP_{\alpha(\zeta)}}
  \name{\bbQ}_{\alpha(\zeta)}=\bbC$ with $\name{c}_\zeta$ being the
  $\bbP_{\alpha(\zeta)+1}$--name for the Cohen real in $\can$ added by
  $\name{\bbQ}_{\alpha(\zeta)}$,
\item[(e)] $\name{\tau}_\zeta$ is a $\bbP_\alpha$--name for an element of 2
  (for $\zeta<\mu$),
\item[(f)] for $\zeta<\mu$,  $\name{\psi}_\zeta$ is a
  $\bbP_\alpha$--name for a null slalom below $f$ such that
\[\forces_{\bbP_\alpha}\mbox{ `` }\name{\psi}_\zeta(i)=
\left\{\begin{array}{rl}
\name{\varphi}_\zeta(i)&\mbox{ if }\name{c}_\zeta(i)=\name{\tau}_\zeta,\\
\emptyset        &\mbox{ if }\name{c}_\zeta(i)=1-\name{\tau}_\zeta
\end{array}\right. \quad \mbox{ for each }i<\omega\mbox{ '',}\]
and $\name{\bar{\psi}}= \langle\name{\psi}_\zeta:\zeta<\mu\rangle$ is the
resulting $\bbP_\alpha$--name for a $\mu$--sequence of null slaloms below
$f$.
\end{enumerate}
Then $\langle \bbP_\beta,\name{\bbQ}_\beta:\beta<\alpha\rangle\conc
\langle\bbP_\alpha, \qforname\rangle$ is a Y--iteration of length
$\alpha+1$.
\end{lemma}

\begin{proof}
First we consider the case when $\mu=\kappa$ and let us argue that
\ref{Ydef}(1)$(\otimes)_3$ holds for $\bbP_{\alpha+1}$.

Let $n<\omega$, $p_\zeta\in \bbP_{\alpha+1}$ and $\delta_\zeta<\kappa$ (for
$\zeta<\kappa$) be such that $\delta_\zeta\neq\delta_{\zeta'}$ for
$\zeta<\zeta'<\kappa$.  For each $\zeta<\kappa$ pick a condition
$p'_\zeta\in \bbP_{\alpha+1}$ stronger than $p_\zeta$ and such that
\begin{enumerate}
\item[$(*)_1$]  $\alpha\in\dom(p'_\zeta)$ and for some  $k^\zeta,
  m^\zeta,u^\zeta$ and  $\sigma^\zeta$ (objects, not names) we have
  $p'_\zeta\rest\alpha\forces_{\bbP_\alpha} \mbox{`` } p'_\zeta(\alpha)=
  (k^\zeta,m^\zeta,u^\zeta,\sigma^\zeta)$ ''.
\end{enumerate}
Choose conditions $p_\zeta''\in\bbP_{\alpha+1}$ stronger than $p_\zeta'$
(so also $p''_\zeta\geq p_\zeta$) and such that $p''_\zeta(\alpha)=
p'_\zeta(\alpha)$ and for all $\zeta$:
\begin{enumerate}
\item[$(*)_2$] for some (objects, not names)  $\gt^\zeta_\vare$ for
  $\vare\in u^\zeta$ we have $p_\zeta''\rest\alpha\forces_{\bbP_\alpha}
  \mbox{`` } \name{\tau}_\vare=\gt^\zeta_\vare$ '',
\item[$(*)_3$] for  some $i^\zeta<\omega$ for all $\vare\in u^\zeta$ we have
  that
\[\alpha(\vare)\in\dom(p''_\zeta)\quad\mbox{ and }\quad
p_\zeta''(\alpha(\vare))\in {}^{\textstyle i^\zeta}2 \mbox{ are
  actual objects, not names.}\]
\end{enumerate}
Since each $\name{\varphi}_\vare$ is a $\bbP_{\alpha(0)}$--name, we may
decide the initial segments of $\name{\varphi}_\vare$ by strengthening
$p_\zeta''\rest \alpha(0)$ only (i.e., without changing $p_\zeta''\rest
[\alpha(0),\alpha]$). Therefore, after using a procedure similar to that in
the proof of \ref{forprop}, for each $\zeta<\kappa$ we  may find a
condition $p^*_\zeta\in\bbP_{\alpha+1}$, $K^\zeta>k^\zeta+i^\zeta$ and a
sequence $\rho^\zeta\in\prod\limits_{i<K^\zeta} {\mathcal P}(f(i))$ such
that
\begin{enumerate}
\item[$(*)_4$] $p_\zeta\leq p_\zeta''\leq p^*_\zeta$, and  $p_\zeta''\rest
  [\alpha(0),\alpha)= p_\zeta^*\rest [\alpha(0),\alpha)$, and
\item[$(*)_5$] $p^*_\zeta\forces_{\bbP_\alpha}$`` $p^*_\zeta(\alpha)=
  (K^\zeta,m^\zeta,u^\zeta,\rho^\zeta)$ ''.
\end{enumerate}
Next we may find a set $S\subseteq \kappa$ of size $\kappa$ and $K^*, m^*,
\rho^*, i^*$ and $\ell^*$ such that
\begin{enumerate}
\item[$(*)_6$] $K^*=K^\zeta$, $m^*=m^\zeta$,  $\rho^*=\rho^\zeta$,
  $|u^\zeta|=\ell^*$ and $i^\zeta=i^*$ for all $\zeta\in S$,
\item[$(*)_7$] $\{u^\zeta:\zeta\in S\}$ is a $\Delta$--system of
  finite subsets of $\kappa$ with root $U$,
\item[$(*)_8$] $\{\dom(p^*_\zeta):\zeta\in S\}$ is a $\Delta$--system of
  finite subsets of $\alpha+1$ with root $D$,
\item[$(*)_9$] for some $\vare^*<\kappa$ we have
  $D\setminus\{\alpha\}\subseteq \alpha (\vare^*)$ and $U=u^\zeta
  \cap\vare^*$ for all $\zeta\in S$.
\end{enumerate}
Since $\langle\bbP_\beta,\name{\bbQ}_\beta:\beta<\alpha(\vare^*) \rangle$ is
a Y--iteration, we may apply \ref{Ydef}(1)$(\otimes)_3$ to $\langle
p_\zeta^*\rest \alpha(\vare^*), \delta_\zeta:\zeta\in S\rangle$ and
$n$. This will give us $v\subseteq S$, $q_0\in\bbP_{\alpha(\vare^*)}$, $m>n$
and  $A_\zeta \in [f(m)]^{\textstyle  g(m)}$ for $\zeta\in v$ such that
\begin{enumerate}
\item[$(*)_{10}$]
\begin{itemize}
\item $|v|\geq\frac{f(m)}{2\cdot g(m)}$ and $p_\zeta^*\rest
  \alpha(\vare^*)\leq q_0$ for all $\zeta\in v$, and
\item $A_\zeta\cap A_{\zeta'}=\emptyset$ for distinct $\zeta,\zeta'\in v$,
  and
\item $q_0\forces_{\bbP_{\alpha(\vare^*)}}\mbox{`` }\big(\forall \zeta\in
  v\big)\big(\name{\nu}(\delta_\zeta)(m)= A_\zeta\big)\mbox{ ''}$.
\end{itemize}
\end{enumerate}
Next, since $\name{\varphi}_\vare$ are $\bbP_{\alpha(0)}$--names, we
may we pick $q_1\in\bbP_{\alpha(\vare^*)}$, $q_1\geq q_0$, $K>K^*\geq
i^*$ and $\rho_\vare\in\prod\limits_{i<K} \cP(f(i))$ (for $\vare\in U$) such
that $q_1\forces_{\bbP_{\alpha(\vare^*)}}\mbox{`` }(\forall \vare\in U)(
\name{\varphi}_\vare\rest K=\rho_\vare) \mbox{ ''}$ and
\[q_1\forces_{\bbP_{\alpha(\vare^*)}}\mbox{`` }\big(\forall j\geq
K\big)\big(\forall\zeta\in v\big)\big(\forall\vare\in u^\zeta\big) \big(|
\name{\varphi}_\vare(j)|<\frac{f(j)}{|v|\cdot \ell^*\cdot
  m^*}\big)\mbox{ ''}\]
Define $q\in\bbP_{\alpha+1}$ so that
\begin{itemize}
\item $\dom(q)=\dom(q_1)\cup\bigcup\{\dom(p_\zeta^*): \zeta\in v\}$,
\item $q\rest \alpha(\vare^*)=q_1$,
\item if $\zeta\in v$ and  $\beta\in \dom(p_\zeta^*)\setminus
  (\alpha(\vare^*)\cup\{\alpha(\vare): \vare\in u^\zeta\})$, then
  $q(\beta)=p_\zeta^*(\beta)$,
\item if $\zeta\in v$ and $\vare\in u^\zeta\setminus \vare^*$, then
  $q(\alpha(\vare))\in {}^{\textstyle K} 2$ is such that $p_\zeta''( \alpha(
  \vare))= p_\zeta^*(\alpha(\vare))\vtl q(\alpha(\vare))$ and for $i\in
  [i^*,K)$ we have $q(\alpha(\vare))(i)=1-\gt_\vare^\zeta$,
\item $q(\alpha)=(K,m^*,u^+,\sigma^+)$, where $u^+=\bigcup\{u^\zeta:\zeta\in
  v\}$ and $\sigma^+\in \prod\limits_{i<k^+} {\mathcal P}(f(i))$ is such that
  $\rho^*\vtl \sigma^+$  and $\sigma^+(i)=\bigcup\limits_{\vare\in U}
  \rho_\vare(i)$ for $i\in [K^*,K)$.
\end{itemize}
The rest, when $\mu=\kappa$, should be clear.
\medskip

Let us assume now that $\mu<\kappa$ and again, to argue for
\ref{Ydef}(1)$(\otimes)_3$, suppose that $n<\omega$, $p_\zeta\in
\bbP_{\alpha+1}$ and $\delta_\zeta<\kappa$ (for $\zeta<\kappa$) are such
that $\delta_\zeta\neq\delta_{\zeta'}$ for $\zeta<\zeta'<\kappa$. Passing to
stronger conditions we may assume that, for each $\zeta<\kappa$,
\[\alpha\in \dom(p_\zeta)\quad\mbox{ and }\quad p_\zeta\rest\alpha
\forces_{\bbP_\alpha} \mbox{`` } p_\zeta(\alpha)=(k^\zeta,m^\zeta,u^\zeta,
\sigma^\zeta)\mbox{ ''}\]
(where $k^\zeta,m^\zeta,u^\zeta,\sigma^\zeta$ are actual objects). For some
$\vare^*<\mu$ and $k,m,u,\sigma$ the set
\[S=\big\{\zeta<\kappa: \dom(p_\zeta)\subseteq\alpha(\vare^*) \cup
\{\alpha\}\ \&\ (k^\zeta,m^\zeta,u^\zeta,\sigma^\zeta)=(k,m,u,
\sigma)\big\}\]
is of size $\kappa$. Like before, $\langle\bbP_\beta,\name{\bbQ}_\beta:
\beta< \alpha(\vare^*) \rangle$ is a Y--iteration, so we may find
$v\subseteq S$, $q_0\in\bbP_{\alpha(\vare^*)}$, $m>n$ and  $A_\zeta \in
[f(m)]^{\textstyle  g(m)}$ for $\zeta\in v$ such that demands listed in
$(*)_{10}$ are satisfied. Let $q\in\bbP_{\alpha+1}$ be such that $\dom(q)=
\dom(q_0)\cup\{\alpha\}$ and $q\rest\alpha\forces q(\alpha)=(k,m,u,\sigma)$.
\end{proof}

\begin{theorem}
\label{main}
Assume $\kappa<\lambda$ are uncountable regular cardinals such that
$\lambda^\kappa=\lambda$. Let $H\subseteq \bbY^\lambda_\kappa$ be generic
over $\bV$ and let $\bar{\bbQ}=\langle\bbP_\alpha,\name{\bbQ}_\alpha:
\alpha<\lambda\rangle=\bigcup H\in \bV[H]$ and
$\bbP_\lambda=\lim(\bar{\bbQ})$. Then $\bbP_\lambda$ is  a ccc forcing
notion with a dense subset of size $\lambda$ and
\[\begin{array}{ll}
\forces_{\bbP_\lambda}\mbox{`` }&{\bf MA}_{{<}\kappa}({\rm ccc})\
\mbox{ and }\ {\bf MA}(\mbox{\rm $\sigma$--centered})\ \mbox{ and}\\
&\add(\cI_f)=\kappa\ \mbox{ and }\ \cofin^-(\cI_f)\geq\kappa^+\
\mbox{ and }\ 2^{\aleph_0}=\lambda\mbox{ ''.}
\end{array}\]
\end{theorem}

\begin{proof}
First note that the forcing with $\bbY^\lambda_\kappa$ does not add
sequences of ordinals of length ${<}\lambda$ (by Lemma
\ref{manyYiter}(5)). Hence in $\bV[H]$ we still have that $\kappa,\lambda$
are regular cardinals and $\lambda^\kappa=\lambda$.

Let us work in $\bV[H]$.

Clearly $\bar{\bbQ}$ is a Y--iteration for $\kappa,\lambda$ of length
$\lambda$. Hence $\bbP_\lambda$ is a ccc forcing notion, it has a dense
subset of size $\lambda$ and forces that $2^{\aleph_0}=\lambda$ (remember
\ref{Ydef}(1)$(\otimes)_2$, \ref{manyYiter}(2)). {\em A canonical
  $\bbP_\lambda$--name\/} $\name{\eta}$ for a real in
$\prod\limits_{n<\omega}\cZ_n$ (where $\langle \cZ_n:n<\omega\rangle
\in\bV$, $\cZ_n\neq \emptyset$) is a sequence $\langle A_n,\pi_n:n<
\omega\rangle$ such that each $A_n$ is a maximal antichain in $\bbP_\lambda$,
$\pi_n:A_n\longrightarrow\cZ_n$ and $q\forces_{\bbP_\lambda}$``
$\name{\eta}(n)=\pi_n(q)$ '' for $q\in A_n$, $n<\omega$. For every
$\bbP_\lambda$--name $\name{\rho}$ for an element of
$\prod\limits_{n<\omega}\cZ_n$ there is a canonical name
$\name{\eta}$ such that $\forces \name{\eta}=\name{\rho}$. Also, if
$\name{\eta}$ is a canonical $\bbP_\lambda$--name for a real, then it is a
$\bbP_\alpha$--name for some $\alpha<\lambda$.
\medskip

Let us argue that $\forces_{\bbP_\lambda}\cofin^-(\cI_f)\geq \kappa^+$. If
not, then for some infinite regular cardinal $\mu\leq\kappa$ and
$\bbP_\lambda$--names $\name{\cI}_\zeta,\name{\varphi}_\zeta$ (for
$\zeta<\mu$) we have
\begin{enumerate}
\item[$(\circledast)_1$]
$\displaystyle\forces_{\bbP_\lambda}\mbox{`` }\name{\varphi}_\zeta\in \cS_f\
\mbox{ and } \ \name{\cI}_\zeta\subseteq \cI_f$ is an ideal '',
\end{enumerate}
and for some $p\in\bbP_\lambda$
\begin{enumerate}
\item[$(\circledast)_2$]
$\displaystyle p\forces_{\bbP_\lambda}\mbox{`` }\bigcup_{\zeta<\mu}
\name{\cI}_\zeta=\cI_f\ \mbox{ and }\ (\forall\zeta<\mu)([
\name{\varphi}_\zeta] \notin \name{\cI}_\zeta)$ ''.
\end{enumerate}
We may assume that all $\name{\varphi}_\zeta$ are $\bbP_{\alpha_0}$--names
for some $\alpha_0<\lambda$.

Suppose now that $\zeta<\mu$ and $\name{c}_\zeta$ is a canonical
$\bbP_\lambda$--name for a real in $\can$. Let $\name{\psi}^0_\zeta,
\name{\psi}^1_\zeta$ be $\bbP_\lambda$--names for elements of $\cS_f$ such
that for each $n<\omega$, $i<2$ we have
\[\forces_{\bbP_\lambda}\mbox{ `` }\name{\psi}_\zeta^i(n)=
\left\{\begin{array}{rl}
\name{\varphi}_\zeta(n)&\mbox{ if }\name{c}_\zeta(n)=i,\\
\emptyset        &\mbox{ if }\name{c}_\zeta(n)=1-i
\end{array}\right. \quad \mbox{ for each }n<\omega\mbox{ ''.}\]
Then $\forces_{\bbP_\lambda} [\name{\varphi}_\zeta] =[\name{\psi}^0_\zeta]
\cup [\name{\psi}^1_\zeta]$, so $p
\forces_{\bbP_\lambda}$``$[\name{\psi}^0_\zeta] \notin \name{\cI}_\zeta$ or
$[\name{\psi}^1_\zeta]\notin \name{\cI}_\zeta$ ''. Let
$\name{\tau}=\tau(\zeta,\name{c}_\zeta)$ be a canonical $\bbP_\lambda$--name
for a member of $\{0,1\}$ such that $p\forces$`` $[\name{\psi}^{
  \name{\tau}}_\zeta] \notin \name{\cI}_\zeta$ ''.

\begin{claim}
\label{clA}
For some sequence $\langle\alpha(\zeta),\name{c}_\zeta,\name{\psi}_\zeta:
\zeta<\mu\rangle$ we have
\begin{enumerate}
\item[(i)] $\langle\alpha(\zeta):\zeta<\mu\rangle\subseteq\lambda$ is
  strictly increasing with $\alpha_0\leq\alpha(0)$, and for each $\zeta<\mu$:
\item[(ii)] $\forces_{\bbP_{\alpha(\zeta)}}\name{\bbQ}_{\alpha(\zeta)}=\bbC$
  and $\name{c}_\zeta$ is the canonical $\bbP_{\alpha(\zeta)+1}$--name for
  the Cohen real in $\can$ added by $\name{\bbQ}_{\alpha(\zeta)}$,  and
  $\tau(\zeta,\name{c}_\zeta)$ is a $\bbP_{\alpha(\zeta+1)}$--name (for a
  member of $\{0,1\}$),
\item[(iii)] $\name{\psi}_\zeta$ is a $\bbP_{\alpha(\zeta+1)}$--name for an
  element of $\cS_f$ such that
\[\forces_{\bbP_{\alpha(\zeta+1)}}\mbox{ `` }\name{\psi}_\zeta(n)=
\left\{\begin{array}{rl}
\name{\varphi}_\zeta(n)&\mbox{ if }\name{c}_\zeta(n)=\tau(\zeta,
\name{c}_\zeta),\\
\emptyset  &\mbox{ if }\name{c}_\zeta(n)=1-\tau(\zeta,\name{c}_\zeta)
\end{array}\right. \quad \mbox{ for all }n<\omega\mbox{ '',}\]
\item[(iv)] if $\alpha^*=\sup(\alpha(\zeta):\zeta<\mu)$, then
  $\forces_{\bbP_{\alpha^*}} \name{\bbQ}_{\alpha^*}=\qforname$, where
  $\name{\bar{\psi}} = \langle\name{\psi}_\zeta:\zeta<\mu\rangle$.
\end{enumerate}
\end{claim}

\begin{proof}[Proof of the Claim]
We move back to $\bV$ and we use a density argument in
$\bbY^\lambda_\kappa$ above $P= \langle \bbP_\beta,\name{\bbQ}_\beta:
\beta<\alpha_0+1\rangle\in \bbY^\lambda_\kappa$. Let $\name{T}$ be a
$\bbY^\lambda_\kappa$--name for the function $\tau(\cdot,\cdot)$ introduced
(in $\bV[H]$) earlier. Note that if $\name{c}$ is a canonical
$\bbP_\gamma^*$--name, $Q^*=\langle\bbP_\beta^*,\name{\bbQ}^*_\beta:
\beta<\gamma\rangle\in\bbY^\lambda_\kappa$  and $\zeta<\mu$, then $Q^*$
forces that $(\zeta,\name{c})$ belongs to the domain of $\name{T}$ and
$\name{T}(\zeta,\name{c})$ is a $\bbY^\lambda_\kappa$--name for an element
of $\bV$.

Suppose that $Q=\langle\bbP_\beta' ,\name{\bbQ}_\beta':\beta< \alpha \rangle
\in \bbY^\lambda_\kappa$ is a condition stronger than $P$ (so $\alpha_0+1
\leq \alpha$ and $\name{\bbQ}_\beta'= \name{\bbQ}_\beta$ for
$\beta\leq\alpha_0$).

By induction on $\zeta<\mu$ we build a sequence $\langle Q_\zeta,
\alpha(\zeta),\name{c}_\zeta: \zeta<\mu\rangle$ such that
\begin{enumerate}
\item[$(\boxtimes)_1$] $Q_\zeta=\langle\bbP_\beta',\name{\bbQ}_\beta':
  \beta\leq\alpha(\zeta)\rangle\in\bbY^\lambda_\kappa$ (so
  $\lh(Q_\zeta)=\alpha(\zeta)+1<\lambda$),
\item[$(\boxtimes)_2$] for $\zeta<\vare<\mu$ we have
\[Q\leq _{\bbY^\lambda_\kappa} Q_\zeta\leq _{\bbY^\lambda_\kappa}
Q_\vare\quad \mbox{ and }\quad \alpha<\alpha(\zeta)< \alpha(\vare)<
\lambda,\]
\item[$(\boxtimes)_3$] $\forces_{\bbP_{\alpha(\zeta)}'} \name{\bbQ}_{\alpha
      (\zeta)}' =\bbC$ and $\name{c}_\zeta$ is the canonical
    $\bbP_{\alpha(\zeta)+1}'$--name for the Cohen real in $\can$ added by
    $\name{\bbQ}_{\alpha(\zeta)}'$,
\item[$(\boxtimes)_4$] $Q_{\zeta+1}$ decides the value of $\name{T}(\zeta,
  \name{c}_\zeta)$ and forces (in $\bbY^\lambda_\kappa$) that it is a
  $\bbP_{\alpha(\zeta+1)}'$--name.
\end{enumerate}
The construction is clearly possible by Lemma \ref{manyYiter}(2,4). Then
letting $\alpha^*=\sup(\alpha(\zeta):\zeta<\mu)$ we have that $Q_\mu=
\langle \bbP_\beta,\name{\bbQ}_\beta': \beta<\alpha^*\rangle\in
\bbY^\lambda_\kappa$  is a condition stronger than all $Q_\zeta$ (for
$\zeta<\mu$); remember \ref{manyYiter}(4) again. Moreover, if names
$\name{\psi}_\zeta$ are defined as in clause (iii), and $\name{\tau}_\zeta$
is the value forced to $\name{T}(\zeta,\name{c}_\zeta)$ by $Q_{\zeta+1}$
(see $(\boxtimes)_4$ above) and $\name{c}_\zeta$ are as described in
$(\boxtimes)_3$, then the assumptions of Lemma \ref{onemore} are
satisfied. Therefore $Q^*=Q_\mu\conc\langle \bbP_{\alpha^*},\bbQ^*_\mu
(\name{\bar{\psi}})\rangle\in \bbY^\lambda_\kappa$ is a condition stronger
than $Q$. This condition forces in $\bbY^\lambda_\kappa$ that $\langle
\alpha(\zeta),\name{c}_\zeta,\name{\psi}_\zeta: \zeta<\mu\rangle$ satisfies
the demands (i)--(iv).
\end{proof}

Let $\alpha(\zeta),\name{c}_\zeta,\name{\psi}_\zeta$ (for $\zeta<\mu$) and
$\alpha^*$ be in Claim \ref{clA}(i--iv), so in particular
$\forces_{\bbP_{\alpha^*}} \name{\bbQ}_{\alpha^*} = \qforname$. Let
$\name{\varsigma}$ be a $\bbP_{\alpha^*+1}$--name for the null slalom added
by $\name{\bbQ}_{\alpha^*}$ (see Definition \ref{forforsla}). It follows
from Proposition \ref{forprop} that
\[\forces_{\bbP_\lambda} \bigcup_{\zeta<\mu} [\name{\psi}_\zeta]\subseteq
[\name{\varsigma}]\in \cI_f,\]
and hence, by $(\circledast)_2$, $p\forces_{\bbP_\lambda} \big(\exists\vare<\mu \big) \big(
\bigcup\limits_{\zeta<\mu} [\name{\psi}_\zeta] \in \name{\cI}_\vare
\big)$. Pick $\vare^*<\mu$ and a condition $q\in \bbP_\lambda$ stronger than
$p$ such that  $q\forces_{\bbP_\lambda} \bigcup\limits_{\zeta<\mu} [\name{\psi}_\zeta]
\in \name{\cI}_{\vare^*}$. Then also  $q\forces [\name{\psi}_{\vare^*}]
\in \name{\cI}_{\vare^*}$ (remember $(\circledast)_1$), but this contradicts
$(\circledast)_2$.
\medskip

To argue that $\forces_{\bbP_\lambda}{\bf MA}_{{<}\kappa}({\rm ccc})$
note that every $\bbP_\lambda$--name $\name{\bbQ}$ for a ccc forcing notion
on some $\gamma^*<\kappa$ is actually a $\bbP_\alpha$--name for some
$\alpha<\lambda$. Therefore by the standard density argument in
$\bbY^\lambda_\kappa$, for some $\beta<\lambda$ we have $\forces_{\bbP_\beta}
\name{\bbQ}_\beta=\name{\bbQ}$ (remember \ref{manyYiter}(2)). Similarly we may
justify that $\forces_{\bbP_\lambda}{\bf MA}(\sigma\mbox{--centered})$.
\medskip

It follows from Lemma \ref{cruciallemma} that $\forces_{\bbP_\lambda}
\add(\cI_f)\leq \kappa$ and by $\forces_{\bbP_\lambda}{\bf MA}_{{<}\kappa}
({\rm ccc})$ we see that the equality is forced.
\end{proof}

\begin{corollary}
\label{mcor}
It is consistent that $\add(\cI_f)=\aleph_1$ and $\cofin^-(\cI_f)=
\cofin(\cI_f)=\aleph_2$.
\end{corollary}

\section{Open problems}

Can we get a result parallel to Corollary \ref{mcor} for the null and/or
meager ideals? Or even better:

\begin{problem}
Let $\cI$ be either the meager ideal $\cM$ or the null ideal $\cN$. Is
it consistent that
\[\add(\cI)<\cofin(\cI)<\cof(\cI)\quad\mbox{ ?}\]
\end{problem}

The method used in the proof of \ref{main}, \ref{mcor} gives the consistency
of $\add(\cI_f)\leq\kappa\ \&\ \kappa^+\leq \cofin^-(\cI_f)$. Can the gap be
bigger?

\begin{problem}
Is it consistent that   $\add(\cI_f)=\aleph_\alpha<\aleph_{\alpha+\omega}
\leq \cofin^-(\cI_f)$ ?
\end{problem}

The cardinal invariant $\cofin$ introduced by Borodulin--Nadzieja and
G{\l}ab has several natural relatives (or variants), some were listed in
Definition \ref{coefficients}. Are those coefficients distinct or they are
equivalent within the realm of nice $\sigma$--ideals?

\begin{problem}
Is it consistent that for some Borel $\sigma$--ideal $\cI$ on $\can$ we have
$\cofin^*(\cI)<\cofin^-(\cI)$ ? Or $\cofin^-(\cI)<\cofin(\cI)$ ? Or
$\cofin(\cI)<\cofin^+(\cI)$ ?
\end{problem}


\end{document}